\documentclass{article}

\usepackage[a4paper, margin=1in]{geometry}

\usepackage[english]{babel}

\usepackage{multicol}

\usepackage{setspace}

\usepackage{booktabs}

\usepackage{amssymb}

\usepackage{adjustbox}

\usepackage{graphicx}

\usepackage{amsmath}

\usepackage{mwe}

\usepackage{colortbl}

\usepackage{caption}

\usepackage{mathrsfs}

\usepackage{bbding}

\usepackage{float}

\usepackage{multirow}

\usepackage{stix}

\usepackage{tikz-cd}

\usepackage{hyperref}
\hypersetup{colorlinks=true,citecolor=blue,linkcolor=blue,linktocpage=true,urlcolor=blue}

\usepackage{xcolor}

\usepackage[all,cmtip]{xy}

\usetikzlibrary {chains}
\usetikzlibrary{arrows.meta}

\usepackage{amsthm}
\newtheorem{theorem}{Theorem}
\newtheorem{definition}{Definition}
\newtheorem{proposition}{Proposition}
\newtheorem{lemma}{Lemma}

\usepackage{authblk}

\title{\textbf{A Category of Ordered Algebras Equivalent to the Category of Multialgebras}}
\author{Coniglio, Marcelo E.\thanks{coniglio@unicamp.br} }

\affil{Institute of Philosophy and the Humanities - IFCH and\\
Centre for Logic, Epistemology and The History of Science - CLE\\
University of Campinas - Unicamp\\
Campinas, SP, Brazil}

\author{Toledo, Guilherme V.\thanks{guivtoledo@gmail.com}}

\affil{Department of Computer Science\\
The Spiegel Mathematics \& Computer Center\\
Bar Ilan University\\
Ramat Gan, Tel Aviv, Israel}

\date{\vspace{-5ex}}

\usepackage{fancyhdr}
\providecommand{\keywords}[1]{\textbf{\textit{Keywords:}} #1}
\begin{document}

\setcounter{page}{1}     

\maketitle

\begin{abstract}
		It is well known that there is a correspondence between sets and complete, atomic Boolean algebras (\textit{CABA}'s) taking a set to its power-set and, reciprocally, a complete, atomic Boolean algebra to its set of atomic elements. Of course, such a correspondence induces an equivalence between the opposite category of $\textbf{Set}$ and the category of \textit{CABA}'s.

We extend this result by taking multialgebras over a signature $\Sigma$, specifically those whose non-deterministic operations cannot return the empty-set, to \textit{CABA}'s with their zero element removed and a structure of $\Sigma$-algebra compatible with its order; reciprocally, one of these ``almost Boolean'' $\Sigma$-algebras is taken to its set of atomic elements equipped with a structure of multialgebra over $\Sigma$. This leads to an equivalence between the category of $\Sigma$-multialgebras and a category of ordered $\Sigma$-algebras. 

The intuition, here, is that if one wishes to do so, non-determinism may be replaced by a sufficiently rich ordering of the underlying structures.
\end{abstract}

\keywords{Multialgebras, ordered algebras, non-deterministic semantics.}

\section*{Introduction}

It is a seminal result (see \cite{Oosten} for a proof) that a correlation between sets and complete, atomic Boolean algebras (\textit{CABA}'s) exists: a set is taken to its power-set, while a \textit{CABA} is taken to its set of atomic elements. These two assignments can be made into functors, giving rise to an equivalence of $\textbf{Set}^{op}$ and $\textbf{CABA}$, the category with \textit{CABA}'s as objects.

This is part of a broader area of study, known as Stone dualities, which studies relationships between posets and topological spaces and was established by Stone (\cite{Stone}) and his representation theorem, which states that every Boolean algebra is isomorphic to a field of sets, specifically the algebra of clopen sets of its Stone space (a topological space where points are ultrafilters of the original Boolean algebra). Of course, this corresponds to an equivalence between the category $\textbf{BA}$ of Boolean algebras and that of Stone spaces.

In the search of dualities with categories of posets, we focus on a more concrete equivalence closely related to the one between $\textbf{Set}^{op}$ and $\textbf{CABA}$. We arrive at a category of multialgebras (originally developed in \cite{Marty}) over a signature $\Sigma$ by adding further structure to $\textbf{Set}$; correspondingly, we replace $\textbf{CABA}$ by a category of \textit{CABA}'s equipped with $\Sigma$-operations compatible with its order.

However, we choose to stress slightly distinct categories, although results as per the aforementioned terms do exist. We are most interested in non-partial multialgebras, where the result of an operation never returns the empty set. Consequently, we exchange \textit{CABA}'s by posets corresponding to power-sets with the empty-set removed (that is, \textit{CABA}'s without minimum elements). This way, a multialgebra, with universe $A$, is taken to an algebra over the set of non-empty subsets of $A$, with order given by inclusion and operations given by ``accumulating'' the operations of the multialgebra, while reciprocally, an ``almost Boolean'' algebra is taken to its set of atomic elements, transformed into a multialgebra.

In the area of research of non-deterministic semantics (\cite{AvronLev}), specially paraconsistent logics (\cite{ParLog}), this offers an alternative: many logicians are reluctant to appeal to multialgebras in order to characterze a given logic, and the equivalence we here present shows one can, if one chooses to, replace such non-deterministic structures for more classically-behaved algebras, with an added underlying order.

In the first section, we give the definition of multialgebras we will use and introduce a brief characterization of power-sets without the empty-set. In the second section, we introduce a naive approach to what we would like to accomplish, and show why it fails. In the third section, we introduce the categories for which our desired result actually holds and the functors that will establish an equivalence between them, equivalence we detail in section four. The final section is reserved for related results.

A preliminary version of this paper can be found in the PhD thesis \cite{Thesis}.

\section{Preliminary notions}

A signature is a collection $\Sigma=\{\Sigma_{n}\}_{n\in \mathbb{N}}$ of possibly empty, disjoint sets indexed by the natural numbers; when there is no risk of confusion, the union $\bigcup_{n\in\mathbb{N}}\Sigma_{n}$ will also be denoted by $\Sigma$.

A $\Sigma-$multialgebra (also known as multialgebra) is a pair $\mathcal{A}=(A, \{\sigma_{\mathcal{A}}\}_{\sigma\in\Sigma})$ such that: $A$ is a non-empty set and, for $\sigma\in\Sigma_{n}$, $\sigma_{\mathcal{A}}$ is a function of the form
\[ \sigma_{\mathcal{A}}:A^{n}\rightarrow\mathcal{P}(A)\setminus\{\emptyset\},\]
where $\mathcal{P}(A)$ denotes the power-set of $A$.

A homomorphism between $\Sigma-$multialgebras $\mathcal{A}=(A, \{\sigma_{\mathcal{A}}\}_{\sigma\in\Sigma})$ and $\mathcal{B}=(B, \{\sigma_{\mathcal{B}}\}_{\sigma\in\Sigma})$ is a function $h:A\rightarrow B$  satisfying, for any $n\in\mathbb{N}$, $\sigma\in\Sigma_{n}$ and $a_{1}, \ldots  , a_{n}\in A$,
\[ \{h(a):  a\in \sigma_{\mathcal{A}}(a_{1}, \ldots  , a_{n})\}\subseteq \sigma_{\mathcal{B}}(h(a_{1}), \ldots  , h(a_{n})).\]
If the inclusion, in the previous equation, were to be replaced by an equality, the resulting $h$ would be a full homomorphism; and a bijective full homomorphism is called an isomorphism. Whenever $h$ is a homomorphism from $\mathcal{A}$ to $\mathcal{B}$, we write $h:\mathcal{A}\rightarrow \mathcal{B}$.

\subsection{Complete, atomic and bottomless Boolean algebras}

Here, we will understand Boolean algebras mostly as partially-ordered sets (\textit{poset}). A pair $(A, \leq)$ is a Boolean algebra if: $\leq$ is reflexive, anti-symmetric and transitive; there are a maximum (denoted by $1$) and a minimum ($0$), which we shall assume distinct; for every pair of elements $(a, b)\in A^{2}$, the set $\{a, b\}$ has a supremum, denoted by $\sup\{a, b\}$ or $a\vee b$, and an infimum, denoted by $\inf\{a, b\}$ or $a\wedge b$; and every element $a$ has a complement $b$ which satisfies
\[ b=\inf\{c\in A:  \sup\{a, c\}=1\}\]
and
\[ b=\sup\{c\in A:  \inf\{a, c\}=0\}.\]

A Boolean algebra $(A, \leq)$ is said to be complete if every $S\subseteq A$ has a supremum. 

\begin{lemma}
\begin{enumerate}
\item Every Boolean algebra $(A, \leq)$ is distributive, meaning $a\vee (b\wedge c)=(a\vee b)\wedge(a\vee c)$ and $a\wedge(b\vee c)=(a\wedge b)\vee(a\wedge c)$ for any $a, b, c\in A$;

\item every complete Boolean algebra $(A, \leq)$ is infinite distributive, meaning that for any $S\cup\{a\}\subseteq A$,\\ $\sup\{\inf\{a, s\}: s\in S\}=\inf\{a, \sup S\}$ and $\inf\{\sup\{a, s\}: s\in S\}=\sup\{a, \inf S\}$.
\end{enumerate}
\end{lemma}

An element $a$ of a Boolean algebra is said to be an atom if it is minimal in the underlying order of the algebra when restricted to $A\setminus\{0\}$, meaning that if $b\leq a$, then either $b=0$ or $b=a$. The set of all the atoms $d$ such that $d \leq a$ will be  denoted by $A_{a}$. Finally, a Boolean algebra is said to be atomic if, for every one of its elements $a$, $a=\sup A_{a}$.

Notice that complete, atomic Boolean algebras are power-sets. Essentially, if one takes for a Boolean algebra $\mathcal{A}=(A, \leq)$ the set $A_{1}$ of all of its atoms, one sees $\mathcal{A}$ is isomorphic to $\mathcal{P}(A_{1})$, the power-set of $A_{1}$, being an arbitrary element $a\in A\setminus\{0\}$ taken, by this isomorphism, to $A_{a}$, while $0$ is taken to $\emptyset$. Reciprocally, one associates to a set $X$ its obvious corresponding complete, atomic Boolean algebra $\mathcal{P}(X)$. For more information, look at Theorem $2.4$ of \cite{Oosten}.

We would like to work with Boolean algebras that are, simultaneously, complete, atomic and bottomless, meaning it lacks a bottom element: this seems a contradiction, given we assume Boolean algebras to have bottom elements, but this can be adequately formalized.

\begin{definition}
Given a non-empty partially ordered set $\mathcal{A}=(A, \leq_{\mathcal{A}})$, we define $\mathcal{A}_{0}$ as the partially ordered set
\[ (A\cup\{0\}, \leq_{\mathcal{A}_{0}}),\]
where we assume $0\notin A$, such that $a\leq_{\mathcal{A}_{0}} b$ if and only if:

\begin{enumerate}
\item either $a\leq_{\mathcal{A}} b$;

\item or $a=0$.
\end{enumerate}
\end{definition}

\begin{definition}
The non-empty partially ordered set $\mathcal{A}$ is a \textit{complete, atomic and bottomless Boolean algebra} whenever $\mathcal{A}_{0}$ is a complete, atomic Boolean algebra.
\end{definition}

Notice that, since $\mathcal{P}(\emptyset)$ only has $\emptyset$ as element, for any complete, atomic and bottomless Boolean algebra $\mathcal{A}$ we cannot have $\mathcal{A}_{0}=\mathcal{P}(\emptyset)$, given $\mathcal{A}$ has at least one element and therefore $\mathcal{A}_{0}$ must have at least two. This means complete, atomic and bottomless Boolean algebras correspond to the powerset of non-empty sets with $\emptyset$ removed.

\begin{proposition}\label{A0 is poset}
If $\mathcal{A}$ is a partially ordered set, so is $\mathcal{A}_{0}$.
\end{proposition}

\begin{lemma}\label{sup of low bounds is low bound}
Given a partially ordered set $(A, \leq)$, for elements $a, b\in A$ we have that the supremum of the lower bounds of $\{a,b\}$, if it exists, is itself a lower bound for $\{a,b\}$.
\end{lemma}

\begin{proof}
Let $s$ be he supremum of the lower bounds of $\{a, b\}$: by definition, this means that for any upper bound $d$ for the set $\{c\in A:  c\leq a, c\leq b\}$ of lower bounds, $s\leq d$; but, since $a$ and $b$ are such upper bounds, we find that $s\leq a$ and $s\leq b$.
\end{proof}

\begin{theorem}\label{list}
A partially ordered set $(A, \leq)$ which satisfies all following conditions is a complete, atomic and bottomless Boolean algebra.

\begin{enumerate}
\item It has a maximum element $1$.

\item All non-empty subsets $S$ of $A$ have a supremum.

\item For every $a\in A$ different from $1$ there exists $b\in A$, named the complement of $a$, such that 
\[ b=\inf\{c\in A:  \sup\{a, c\}=1\}\]
and
\[ b=\sup\{c\in A:  \text{$\inf\{a, c\}$ does not exist}\},\]
property we call being semi-complemented.

\item Denoting by $A_{a}$ the set of minimal elements smaller than $a$, $a=\sup A_{a}$.
\end{enumerate}
\end{theorem}

\begin{proof}
Suppose that $\mathcal{A}=(A, \leq_{\mathcal{A}})$ is a partially ordered set satisfying the previous list of conditions. Since $\mathcal{A}$ is a partially ordered set, so is $\mathcal{A}_{0}$ from Proposition~\ref{A0 is poset}. The maximum $1$ of $\mathcal{A}$ remains a maximum in $\mathcal{A}_{0}$, while $0$ becomes a minimum. For non-zero elements $a$ and $b$, the supremum in $\mathcal{A}_{0}$ remains the same as in $\mathcal{A}$, while if $a=0$ or $b=0$ the supremum is simply the largest of the two. 

If $a$ or $b$ are equal to $0$, the infimum is $0$, while if $a, b\in A$ there are two cases to consider: if $\inf\{a, b\}$ was defined in $\mathcal{A}$, it remains the same in $\mathcal{A}_{0}$; if the infimum was not defined in $\mathcal{A}$, this means that there were no lower bounds for both $a$ and $b$, since otherwise we would have
\[ \inf\{a,b\}=\sup\{c\in A:  c\leq_{\mathcal{A}}a\quad\text{and}\quad c\leq_{\mathcal{A}}b\}\]
by Lemma \ref{sup of low bounds is low bound}, and therefore the infimum of both in $\mathcal{A}_{0}$ is $0$. Every element $a\in A\setminus\{1\}$ already has a complement $b$ in $\mathcal{A}$ such that $b=\inf\{c\in A:  \sup\{a, c\}=1\}$ and
\[ b=\sup\{c\in A:  \text{$\inf\{a, c\}$ does not exist}\};\]
of course the first equality keeps on holding in $\mathcal{A}_{0}$, while the second becomes, remembering that the non-defined infima in $\mathcal{A}$ become $0$ in $\mathcal{A}_{0}$,
\[ b=\sup\{c\in A:  \inf\{a, c\}=0\};\]
the complement of $1$ is clearly $0$ and vice-versa. This proves $\mathcal{A}_{0}$ is a Boolean algebra.

Since $\mathcal{A}$ is closed under suprema of non-empty sets and $\sup\emptyset=0$ in $\mathcal{A}_{0}$, it is clear that $\mathcal{A}_{0}$ is closed under any suprema. Clearly $\mathcal{A}_{0}$ remains atomic, since $\mathcal{A}$ is atomic, what finishes the proof that the previous list of conditions imply $\mathcal{A}$ is a complete, atomic and bottomless Boolean algebra.

\end{proof}

\begin{theorem}
The reciprocal of Theorem \ref{list} holds, meaning that complete, atomic and bottomless Boolean algebras satisfy the list of conditions found in \ref{list}.
\end{theorem}

\begin{proof}
Given a partially ordered set $\mathcal{A}$, suppose $\mathcal{A}_{0}$ is a complete, atomic Boolean algebra.
\begin{enumerate}
\item The maximum $1$ of $\mathcal{A}_{0}$ is still a maximum in $\mathcal{A}$.

\item The supremum of any non-empty set in $\mathcal{A}$ is just its supremum in $\mathcal{A}_{0}$.

\item Given any element $a\neq 1$, its complement $b$ in $\mathcal{A}_{0}$ ends up being also its complement in $\mathcal{A}$. Clearly 
\[ b=\inf\{c\in A:  \sup\{a,c\}=1\}.\]
Now, $\inf\{a,c\}$ does not exist in $\mathcal{A}$ if, and only if, $\inf\{a,c\}=0$ in $\mathcal{A}_{0}$: we already proved that if $\inf\{a,c\}$ does not exist in $\mathcal{A}$ then $\inf\{a,c\}=0$ in $\mathcal{A}_{0}$, remaining to show the reciprocal; if the infimum of $a$ and $c$ existed in $\mathcal{A}$, it would equal $0$ in $\mathcal{A}_{0}$ given the unicity of the infimum, contradicting that $0$ is not in $\mathcal{A}$. This way, we find that in $\mathcal{A}$
\[ b=\sup\{c\in A:  \text{$\inf\{a,c\}$ does not exist}\},\]
as required.

\item Clearly $\mathcal{A}_{0}$ being atomic implies $\mathcal{A}$ being atomic.
\end{enumerate}
\end{proof}

\begin{proposition}\label{sem-inf-dist}
If $(A, \leq_{\mathcal{A}})$ is a complete, atomic and bottomless Boolean algebra, for any $S\subseteq A$, if 
\[ S^{a}=\{s\in S: \text{$\inf\{a, s\}$ exists}\}\neq\emptyset,\]
then
\[ \sup\{\inf\{a, s\}: s\in S^{a}\}=\inf\{a, \sup S\};\]
if $S^{a}=\emptyset$, $\inf\{a, \sup S\}$ also does not exist.
\end{proposition}

\begin{proof}
If $S^{a}=\emptyset$ this means that $\inf\{a,s\}=0$ for every $s\in S$ in $\mathcal{A}_{0}$, and therefore $\inf\{a, \sup S\}=0$, so that the same infimum no longer exists in $\mathcal{A}$.

If $S^{a}\neq\emptyset$, all infima and suprema in $\sup\{\inf\{a,s\}: s\in S^{a}\}$ and $\inf\{a, \sup S\}$ exist in $\mathcal{A}$ and are therefore equal to their counterparts in $\mathcal{A}_{0}$; given $\sup\{\inf\{a, s\}: s\in S^{a}\}=\sup\{\inf\{a, s\}: s\in S\}$ in $\mathcal{A}_{0}$, since $s\in S\setminus S^{a}$ implies $\inf\{a, s\}=0$, by the infinite-distributivity of $\mathcal{A}_{0}$ one proves the desired result.
\end{proof}

The lesson to be taken from this short exposition is that a complete, atomic and bottomless Boolean algebra is a power-set (of a non-empty set) with the empty-set removed. This will be important to us given our multialgebras cannot return the empty-set as the result of an operation.

\section{A first attempt}\label{A first attempt}

Consider the categories $\textbf{Alg}(\Sigma)$ of $\Sigma$-algebras, with homomorphisms between $\Sigma$-algebras as morphisms, and\\ $\textbf{MAlg}(\Sigma)$ of $\Sigma$-multialgebras, with homomorphisms between $\Sigma$-multialgebras as morphisms.

For simplicity, denote the set of non-empty subsets of $A$, $\mathcal{P}(A)\setminus\{\emptyset\}$, by $\mathcal{P}^{*}(A)$. For a $\Sigma$-multialgebra $\mathcal{A}=(A, \{\sigma_{\mathcal{A}}\}_{\sigma\in\Sigma})$, consider the $\Sigma$-algebra $\mathsf{P}(\mathcal{A})=(\mathcal{P}^{*}(A), \{\sigma_{\mathsf{P}(\mathcal{A})}\}_{\sigma\in\Sigma})$ where, for a $\sigma\in\Sigma_{n}$ and nonempty $A_{1}, \ldots  , A_{n}\subseteq A$,
\[ \sigma_{\mathsf{P}(\mathcal{A})}(A_{1}, \ldots  , A_{n})=\bigcup_{(a_{1}, \ldots  , a_{n})\in A_{1}\times\cdots\times A_{n}}\sigma_{\mathcal{A}}(a_{1}, \ldots  , a_{n}).\]
Again, for simplicity, we may write the previous equation as $\sigma_{\mathsf{P}(\mathcal{A})}(A_{1}, \ldots  , A_{n})=\bigcup\{\sigma_{\mathcal{A}}(a_{1}, \ldots  , a_{n}) : a_{i}\in A_{i}\}$. We also define, for $\mathcal{A}$ and $\mathcal{B}$ two $\Sigma$-multialgebras and a homomorphism $ h :\mathcal{A}\rightarrow\mathcal{B}$, the function $\mathsf{P}( h ):\mathsf{P}(A)\rightarrow\mathsf{P}(B)$ such that, for a $\emptyset\neq A'\subseteq A$, 
\[ \mathsf{P}( h )(A')=\{ h (a)\in B:  a\in A'\}.\]

One could hope that $\mathsf{P}( h )$ is actually a $\Sigma$-homomorphism, perhaps making of $\mathsf{P}$ a functor from $\textbf{MAlg}(\Sigma)$ to $\textbf{Alg}(\Sigma)$, but the following result shows this is usually not to be expected.

\begin{lemma}\label{P is anti-hom}
For $\mathcal{A}$ and $\mathcal{B}$ two $\Sigma$-multialgebras and $ h :\mathcal{A}\rightarrow\mathcal{B}$ a homomorphism, $\mathsf{P}( h )$ satisfies
\[ \mathsf{P}( h )(\sigma_{\mathsf{P}(\mathsf{A})}(A_{1}, \ldots  , A_{n}))\subseteq \sigma_{\mathsf{P}(\mathsf{B})}(\mathsf{P}( h )(A_{1}), \ldots  , \mathsf{P}( h )(A_{n}))\]
for all $\sigma\in\Sigma$ and nonempty $A_{1}, \ldots  , A_{n}\subseteq A$. If $ h $ is a full homomorphism, $\mathsf{P}( h )$ is a homomorphism.

\end{lemma}

\begin{proof}
Given $\sigma\in \Sigma_{n}$ and nonempty $A_{1}, \ldots  , A_{n}\subseteq A$, we have that
\[ \sigma_{\mathsf{P}(\mathcal{B})}(\mathsf{P}( h )(A_{1}), \ldots  , \mathsf{P}( h )(A_{n}))=\bigcup\{\sigma_{\mathcal{B}}(b_{1}, \ldots  , b_{n}) : b_{i}\in \mathsf{P}( h )(A_{i})\}=\]
\[\bigcup\{\sigma_{\mathcal{B}}(b_{1}, \ldots  , b_{n}) : b_{i}\in \{ h (a) : a\in A_{i}\}\}=\bigcup\{\sigma_{\mathcal{B}}( h (a_{1}), \ldots  ,  h (a_{n})) : a_{i}\in A_{i}\},\]
which clearly contains
\[ \bigcup\{\{ h (a):  a\in\sigma_{\mathcal{A}}(a_{1}, \ldots  , a_{n})\} : a_{i}\in A_{i}\}= \{ h (a):  a\in \bigcup\{\sigma_{\mathcal{A}}(a_{1}, \ldots  , a_{n}) : a_{i}\in A_{i}\}\}=\]
\[ \{ h (a):  a\in \sigma_{\mathsf{P}(\mathcal{A})}(A_{1}, \ldots  , A_{n})\}=\mathsf{P}( h )(\sigma_{\mathsf{P}(\mathcal{A})}(A_{1}, \ldots  , A_{n})),\]
so that $\mathsf{P}( h )$ satisfies the required property. 

If $ h $ is a full homomorphism, $\sigma_{\mathcal{B}}( h (a_{1}), \ldots  ,  h (a_{n}))=\{ h (a) : a\in\sigma_{\mathcal{A}}(a_{1}, \ldots  , a_{n})\}$, and the inclusion in the equations above becomes an equality.
\end{proof}

So, let us restrict $\mathsf{P}$ for a moment to the category $\textbf{MAlg}_{=}(\Sigma)$, of $\Sigma$-multialgebras with only full homomorphisms between them as morphisms, and let us call this new transformation $\mathsf{P}_{=}:\textbf{MAlg}_{=}(\Sigma)\rightarrow \textbf{Alg}(\Sigma)$.

\begin{proposition}
$\mathsf{P}_{=}$ is, in fact, a functor from $\textbf{MAlg}_{=}(\Sigma)$ to $\textbf{Alg}(\Sigma)$.
\end{proposition}

Unfortunately, $\mathsf{P}_{=}$ is not injective on objects: take the signature $\Sigma_{s}$ with a single unary operator $s$, and consider the $\Sigma$-multialgebras $\mathcal{A}=(\{0,1\}, \{s_{\mathcal{A}}\})$ and $\mathcal{B}=(\{0,1\}, \{s_{\mathcal{B}}\})$ such that: $s_{\mathcal{A}}(0)=s_{\mathcal{A}}(1)=\{1\}$ and $s_{\mathcal{B}}(0)=s_{\mathcal{B}}(1)=\{0,1\}$.

\begin{figure}[H]
    \centering
    \hspace{0,1\textwidth}
    \begin{minipage}[b]{.35\textwidth}
        \centering
         \xymatrix{0 \ar[r]^{s_{\mathcal{A}}} & 1 \ar@(ur, dr)^{s_{\mathcal{A}}}}
    \end{minipage}%
    \hspace{0.2\textwidth}
    \begin{minipage}[b]{.35\textwidth}
    \centering
        \centering
        \xymatrix{0 \ar@/ ^2pc/ [rr]^{s_{\mathcal{B}}} \ar@(dl, ul)^{s_{\mathcal{B}}} & & 1 \ar@/ ^2pc/ [ll]^{s_{\mathcal{B}}} \ar@(ur, dr)^{s_{\mathcal{B}}}}
    \end{minipage}
    \begin{minipage}[t]{.5\textwidth}
        \caption*{The $\Sigma_{s}$-multialgebra $\mathcal{A}$}
    \end{minipage}%
    \begin{minipage}[t]{.4\textwidth}
        \caption*{The $\Sigma_{s}$-multialgebra $\mathcal{B}$}
    \end{minipage}%
\end{figure}

Clearly the two of then are not isomorphic, given that the result of an operation in $\mathcal{A}$ always has cardinality $1$ and in $\mathcal{B}$ alway has cardinality $2$.

However, we have that $s_{\mathsf{P}_{=}(\mathcal{A})}(\{0\})=s_{\mathsf{P}_{=}(\mathcal{A})}(\{1\})=s_{\mathsf{P}_{=}(\mathcal{A})}(\{0, 1\})=\{1\}$, while $s_{\mathsf{P}_{=}(\mathcal{B})}(\{0\})=s_{\mathsf{P}_{=}(\mathcal{A})}(\{1\})=s_{\mathsf{P}_{=}(\mathcal{A})}(\{0, 1\})=\{0,1\}$.

\begin{figure}[H]
    \centering
    \begin{minipage}[b]{.4\textwidth}
        \centering
         \xymatrix{ & \{0,1\} \ar[dr]^{s_{\mathcal{P}^{*}(\mathcal{A})}} & \\ \{0\}\ar[rr]^{s_{\mathcal{P}^{*}(\mathcal{A})}} & & \{1\} \save !R(.7) \ar@(ur, dr)^{s_{\mathcal{P}^{*}(\mathcal{A})}} \restore}
    \end{minipage}%
    \hspace{0.1\textwidth}
    \begin{minipage}[b]{.5\textwidth}
    \centering
        \centering
        \xymatrix{ & \{0,1\} \ar@(ur, ul)_{s_{\mathcal{P}^{*}(\mathcal{B})}} & \\ \{0\} \ar[ur]^{s_{\mathcal{P}^{*}(\mathcal{B})}} & & \{1\} \ar[ul]_{s_{\mathcal{P}^{*}(\mathcal{B})}}}
    \end{minipage}
    \begin{minipage}[t]{.5\textwidth}
        \caption*{The $\Sigma_{s}$-algebra $\mathsf{P}_{=}(\mathcal{A})$}
    \end{minipage}%
    \hspace{0.05\textwidth}
    \begin{minipage}[t]{.4\textwidth}
        \caption*{The $\Sigma_{s}$-algebra $\mathsf{P}_{=}(\mathcal{B})$}
    \end{minipage}%
\end{figure}

Taking the function $ h :\mathcal{P}^{*}(A)\rightarrow\mathcal{P}^{*}(B)$ such that $ h (\{0\})=\{0\}$, $ h (\{1\})=\{0,1\}$, and $ h (\{0,1\})=\{1\}$, we see that it is a bijection and a homomorphism, and therefore $ h :\mathsf{P}_{=}(\mathcal{A})\rightarrow\mathsf{P}_{=}(\mathcal{B})$ is an isomorphism.

\section{An improvement}

The problem with our definition of $\mathsf{P}_{=}$ is that it disregards the structure of the universe of $\mathcal{P}(\mathcal{A})$. So, we change our target category to reflect this structure.

\begin{definition}
Given a signature $\Sigma$, a $(\Sigma, \leq)$-algebra $\mathcal{A}$ is a triple $(A, \{\sigma_{\mathcal{A}}\}_{\sigma\in\Sigma}, \leq_{\mathcal{A}})$ such that:

\begin{enumerate}
\item $(A, \{\sigma_{\mathcal{A}}\}_{\sigma\in\Sigma})$ is a $\Sigma$-algebra;

\item $(A,\leq_{\mathcal{A}})$ is a complete, atomic and bottomless Boolean algebra;

\item if $A_{a}$ is the set of minimal elements of $(A, \leq_{\mathcal{A}})$ (atoms) less than or equal to $a$, for all $\sigma\in\Sigma_{n}$ and $a_{1}, \ldots  , a_{n}$ we have that
\[ \sigma_{\mathcal{A}}(a_{1}, \ldots  , a_{n})=\sup\{\sigma_{\mathcal{A}}(b_{1}, \ldots  , b_{n}):  (b_{1}, \ldots  , b_{n})\in A_{a_{1}}\times\cdots\times A_{a_{n}}\}.\]
\end{enumerate}
\end{definition}

\begin{proposition}
For $\mathcal{A}$ a \textit{$(\Sigma, \leq)$-algebra}, any $\sigma\in\Sigma_{n}$ and $a_{1}, \ldots  , a_{n}, b_{1}, \ldots  , b_{n}\in A$ such that $a_{1}\leq_{\mathcal{A}} b_{1}$, \ldots  , $a_{n}\leq_{\mathcal{A}} b_{n}$, one has $\sigma_{\mathcal{A}}(a_{1}, \ldots  , a_{n})\leq_{\mathcal{A}} \sigma_{\mathcal{A}}(b_{1}, \ldots  , b_{n})$.
\end{proposition}

\begin{proof}
Since, for every $i\in\{1, \ldots  , n\}$, $a_{i}\leq_{\mathcal{A}}b_{i}$, we have that $A_{a_{i}}\subseteq A_{b_{i}}$, one concludes that $A_{a_{1}}\times\cdots\times A_{a_{n}}\subseteq A_{b_{1}}\times\cdots\times A_{b_{n}}$; this way,
\[ \sigma_{\mathcal{A}}(a_{1}, \ldots  , a_{n})=\sup\{\sigma_{\mathcal{A}}(c_{1}, \ldots  , c_{n}):  c_{i}\in A_{a_{i}}\}\leq_{\mathcal{A}} \sup\{\sigma_{\mathcal{A}}(c_{1}, \ldots  , c_{n}):  c_{i}\in A_{b_{i}}\}=\sigma_{\mathcal{A}}(b_{1}, \ldots  , b_{n}).\]
\end{proof}

For a $\Sigma$-multialgebra $\mathcal{A}=(A, \{\sigma_{\mathcal{A}}\}_{\sigma\in\sigma})$, we define $\mathbb{P}(\mathcal{A})$ as the $(\Sigma, \leq)$-algebra 
\[ (\mathcal{P}^{*}(A), \{\sigma_{\mathbb{P}(\mathcal{A})}\}_{\sigma\in\Sigma}, \leq_{\mathbb{P}(\mathcal{A})})\]
such that $(\mathcal{P}^{*}(A), \{\sigma_{\mathbb{P}(\mathcal{A})}\}_{\sigma\in\Sigma})$ is exactly the $\Sigma$-algebra $\mathsf{P}(\mathcal{A})$ defined at the beginning of Section \ref{A first attempt} and, for nonempty subsets $A_{1}$ and $A_{2}$ of $A$, $A_{1}\leq_{\mathbb{P}(\mathcal{A})} A_{2}$ if and only if $A_{1}\subseteq A_{2}$. Since:

\begin{enumerate}
\item $\mathsf{P}(\mathcal{A})$ is a $\Sigma$-algebra;

\item $(\mathcal{P}^{*}(A), \leq_{\mathbb{P}(\mathcal{A})})$ is a complete, atomic and bottomless Boolean algebra, given that $\mathcal{P}(A)$ is a complete, atomic Boolean algebra with at least two elements;

\item and, for $\sigma\in\Sigma_{n}$ and $\emptyset\neq A_{1}, \ldots  , A_{n}\subseteq A$, since the atoms of $A_{i}$ are exactly $A_{A_{i}}=\{\{a\}:  a\in A_{i}\}$,
\[ \sigma_{\mathbb{P}(\mathcal{A})}(A_{1}, \ldots  , A_{n})=\bigcup\{\sigma_{\mathcal{A}}(a_{1}, \ldots  , a_{n}) : a_{i}\in A_{i}\}=\bigcup\{\sigma_{\mathbb{P}(\mathcal{A})}(\{a_{1}\}, \ldots  , \{a_{n}\}) : \{a_{i}\}\in A_{A_{i}}\};\]
\end{enumerate}
we truly have that $\mathbb{P}(\mathcal{A})$ is a $(\Sigma, \leq)$-algebra.

\begin{definition}
Given $(\Sigma, \leq)$-algebras $\mathcal{A}=(A, \{\sigma_{\mathcal{A}}\}_{\sigma\in\Sigma}, \leq_{\mathcal{A}})$ and $\mathcal{B}=(B, \{\sigma_{\mathcal{B}}\}_{\sigma\in\Sigma}, \leq_{\mathcal{B}})$, a function $ h :A\rightarrow B$ is said to be a $(\Sigma, \leq)$-homomorphism, in which case we write $ h :\mathcal{A}\rightarrow\mathcal{B}$, when:

\begin{enumerate}
\item for all $\sigma\in\Sigma_{n}$ and $a_{1}, \ldots  , a_{n}\in A$ we have that 
\[  h (\sigma_{\mathcal{A}}(a_{1}, \ldots  , a_{n}))\leq_{\mathcal{B}}\sigma_{\mathcal{B}}( h (a_{1}), \ldots  ,  h (a_{n}));\]

\item $ h $ is continuous, meaning that, for every non-empty subset $A'\subseteq A$, $ h (\sup A')=\sup\{ h (a):  a\in A'\}$;

\item $ h $ maps minimal elements of $(A, \leq_{\mathcal{A}})$ to minimal elements of $(B, \leq_{\mathcal{B}})$.
\end{enumerate}
\end{definition}

Notice that a $(\Sigma, \leq)-$homomorphism is essentially an ``almost $\Sigma-$homomor\-phism'' which is also continuous and minimal-elements-preserving. Notice also that a $(\Sigma, \leq)$-homomorphism is order preserving: if $a\leq_{\mathcal{A}}b$, then $b=\sup\{a,b\}$, and therefore $ h (b)=\sup\{ h (a),  h (b)\}$, meaning that $ h (a)\leq_{\mathcal{B}} h (b)$.

\subsection{$\mathbb{P}$ is a functor}

\begin{lemma}
The composition of $(\Sigma, \leq)$-homomorphisms returns a $(\Sigma, \leq)$-homo\-morphism.
\end{lemma}

\begin{proof}
Take $(\Sigma, \leq)$-algebras $\mathcal{A}$, $\mathcal{B}$ and $\mathcal{C}$, and $(\Sigma, \leq)$-homomorphisms $ h :\mathcal{A}\rightarrow\mathcal{B}$ and $ h' :\mathcal{B}\rightarrow\mathcal{C}$.

\begin{enumerate}
\item $ h' \circ h $ obviously is a function from $A$ to $C$, so let $\sigma\in\Sigma_{n}$ and $a_{1}, \ldots  , a_{n}\in A$: we have that, since both $ h' $ and $ h $ are $(\Sigma, \leq)$-homomorphisms,
\[  h' \circ h (\sigma_{\mathcal{A}}(a_{1}, \ldots  , a_{n}))= h' ( h (\sigma_{\mathcal{A}}(a_{1}, \ldots  , a_{n})))\leq_{\mathcal{C}} h' (\sigma_{\mathcal{B}}( h (a_{1}), \ldots  ,  h (a_{n}))),\]
because $ h' $ is order-preserving and $ h (\sigma_{\mathcal{A}}(a_{1}, \ldots  , a_{n}))\leq_{\mathcal{B}}\sigma_{\mathcal{B}}( h (a_{1}), \ldots  ,  h (a_{n}))$, and
\[  h' (\sigma_{\mathcal{B}}( h (a_{1}), \ldots  ,  h (a_{n})))\leq_{\mathcal{C}}\sigma_{\mathcal{C}}( h' ( h (a_{1})), \ldots  ,  h' ( h (a_{n})))= \sigma_{\mathcal{C}}( h' \circ h (a_{1}), \ldots  ,  h' \circ h (a_{n}))\]
since $ h' $ is an ``almost homomorphism''.

\item Given a non-empty $A'\subseteq A$, we have that $ h (\sup A')=\sup\{ h (a):  a\in A'\}$ and, denoting $\{ h (a):  a\in A'\}$ as $B'$, we have that $ h' (\sup B')=\sup\{ h' (b):  b\in B'\}$; since $\sup B'= h (\sup A')$, we obtain
\[  h' \circ h (\sup A')=\sup\{ h' (b):  b\in B'\}=\sup\{ h' \circ h (a):  a\in A'\},\]
which means that $ h' \circ h $ is continuous.

\item Finally, if $a\in A$ is a minimal element, $ h (a)\in B$ is a minimal element, since $ h $ preserves minimal elements, and for the same reason $ h' \circ h (a)= h' ( h (a))\in C$ remains a minimal element still, and from all of the above $ h' \circ h $ is a $(\Sigma, \leq)$-homomorphism.
\end{enumerate}
\end{proof}

\begin{proposition}
When we take as objects all $(\Sigma, \leq)$-algebras and as morphisms all the $(\Sigma, \leq)$-homomorphisms between them, the resulting object is a category, denoted by $\textbf{Alg}_{\mathsf{B}}(\Sigma)$.
\end{proposition}

So the transformation taking a $\Sigma$-multialgebra $\mathcal{A}$ to $\mathbb{P}(\mathcal{A})$ and a homomorphism $ h :\mathcal{A}\rightarrow\mathcal{B}$ to the $(\Sigma, \leq)$-homomorphism $\mathbb{P}( h ):\mathbb{P}(\mathcal{A})\rightarrow\mathbb{P}(\mathcal{B})$ such that, for an $\emptyset\neq A'\subseteq A$,
\[ \mathbb{P}( h )(A')=\{ h (a)\in B:  a\in A'\},\]
is a functor, of the form $\mathbb{P}:\textbf{MAlg}(\Sigma)\rightarrow \textbf{Alg}_{\mathsf{B}}(\Sigma)$. First we must show that $\mathbb{P}( h )$ is, in fact, an $(\Sigma, \leq)$-homomor-\\phism: given Lemma \ref{P is anti-hom} and the fact that $\mathsf{P}( h )=\mathbb{P}( h )$, we have that $\mathbb{P}( h )$ satisfies the first condition for being a\\ $(\Sigma, \leq)-$homomorphism; and, if $\emptyset \neq A''$ is a subset of $\mathcal{P}(A)$, we have that
\[ \mathbb{P}( h )(\sup A'')=\{ h (a):  a\in \sup A''\}=\{ h (a):  a\in \bigcup A''\}=\bigcup\{\{ h (a):  a\in A'\} : A'\in A''\}=\]
\[\bigcup\{\mathbb{P}( h )(A') : A'\in A''\}=\sup \{\mathbb{P}( h )(A'):  A'\in A''\},\]
what proves the satisfaction of the second condition; for the third condition, we remember that the minimal elements of $(\mathcal{P}^{*}(A), \subseteq)$ are the singletons, that is, sets of the form $\{a\}$ with $a\in A$, and since $\mathbb{P}( h )(\{a\})=\{ h (a)\}$, $\mathbb{P}( h )$ preserves minimal elements.

\begin{theorem}
$\mathbb{P}:\textbf{MAlg}(\Sigma)\rightarrow \textbf{Alg}_{\mathsf{B}}(\Sigma)$ is a functor.
\end{theorem}

\subsection{$\mathbb{P}$ may be seem as part of a monad}

As is the case with the power-set functor, from $\textbf{Set}$ to itself, we may see $\mathbb{P}$, or even $\mathsf{P}$ and $\mathsf{P}_{=}$, as being part of a monad, although some minor modifications are necessary. So, consider the endofunctor $  \tilde{\mathbb{P}}:\textbf{MAlg}(\Sigma)\rightarrow\textbf{MAlg}(\Sigma)$ such that, for a $\Sigma$-multialgebra $\mathcal{A}=(A, \{\sigma_{\mathcal{A}}\}_{\sigma\in\Sigma})$, $\tilde{\mathbb{P}}\mathcal{A}$ is the $\Sigma$-multialgebra with universe $\mathcal{P}^{*}(A)$ and operations given by
\[ \sigma_{\tilde{\mathbb{P}}\mathcal{A}}(A_{1}, \ldots  , A_{n})=\{\{a\}\in \mathcal{P}^{*}(A) : a\in \bigcup\{\sigma_{\mathcal{A}}(a_{1}, \ldots  , a_{n}) : a_{i}\in A_{i}\}\},\]
for $\sigma$ an $n$-ary symbol and $A_{1}$ through $A_{n}$ non-empty subsets of $A$; and for $\Sigma$-multialgebras $\mathcal{A}$ and $\mathcal{B}$, a homomorphism $ h :\mathcal{A}\rightarrow\mathcal{B}$ and a non-empty $A'\subseteq A$, $\tilde{\mathbb{P}} h :\tilde{\mathbb{P}}\mathcal{A}\rightarrow \tilde{\mathbb{P}}\mathcal{B}$ satisfies $\tilde{\mathbb{P}} h (A')=\{ h (a)\in B: a \in A'\}$. Notice that $\tilde{\mathbb{P}}\mathcal{A}$ is almost the same as $\mathsf{P}(\mathcal{A})$, with the difference that in the latter, operations return subsets of $A$, while in the former they return sets of singletons of $A$, whose union is exactly the result of the operation as performed in $\mathsf{P}(\mathcal{A})$.

For the natural transformations to form a monad together with $\tilde{\mathbb{P}}$, we chose the obvious candidates: $\eta:1_{\textbf{MAlg}(\Sigma)}\rightarrow \tilde{\mathbb{P}}$ and $\epsilon:\tilde{\mathbb{P}}\circ \tilde{\mathbb{P}}\rightarrow \tilde{\mathbb{P}}$ given by, for a $\Sigma$-multialgebra $\mathcal{A}$, an element $a$ of $\mathcal{A}$ and a non-empty collection $\{A_{i}\}_{i\in I}$ of non-empty subsets of $A$, $\eta_{\mathcal{A}}(a)=\{a\}$ and $\epsilon_{\mathcal{A}}(\{A_{i}\}_{i\in I})=\bigcup\{A_{i} : i\in I\}$.

\begin{proposition}
For any $\Sigma$-multialgebra $\mathcal{A}$, $\eta_{\mathcal{A}}$ and $\epsilon_{\mathcal{A}}$ are, indeed, homomorphisms.
\end{proposition}

\begin{proposition}
For any $\Sigma$-multialgebras $\mathcal{A}$ and $\mathcal{B}$, and homomorphism $ h :\mathcal{A}\rightarrow\mathcal{B}$, the identities $\tilde{\mathbb{P}} h \circ \eta_{\mathcal{A}}=\eta_{\mathcal{B}}\circ  h $ and $\tilde{\mathbb{P}} h \circ\epsilon_{\mathcal{A}}=\epsilon_{\mathcal{B}}\circ \tilde{\mathbb{P}}\tilde{\mathbb{P}} h $ are satisfied, meaning $\eta$ and $\epsilon$ are natural transformations.
\end{proposition}

\begin{proof}
Let $a$ be an element of $\mathcal{A}$. We have that $\tilde{\mathbb{P}} h \circ\eta_{\mathcal{A}}(a)=\tilde{\mathbb{P}} h (\eta_{\mathcal{A}}(a))$, and since $\eta_{\mathcal{A}}(a)=\{a\}$, we have that $\tilde{\mathbb{P}} h \circ\eta_{\mathcal{A}}(a)=\{ h (a)\}$. Meanwhile, $\eta_{\mathcal{B}}\circ h (a)=\eta_{\mathcal{A}}( h (a))=\{ h (a)\}$, and as stated both expressions coincide.

Now, let $\{A_{i}\}_{i\in I}$ be an element of $\tilde{\mathbb{P}}\tilde{\mathbb{P}}\mathcal{A}$, meaning it is a non-empty set of non-empty subsets of $\mathcal{A}$:\\ $\tilde{\mathbb{P}} h \circ\epsilon_{\mathcal{A}}(\{A_{i}\}_{i\in I})=\tilde{\mathbb{P}} h (\epsilon_{\mathcal{A}}(\{A_{i}\}_{i\in I}))$, and since $\epsilon_{\mathcal{A}}(\{A_{i}\}_{i\in I})=\bigcup\{A_{i} : i\in I\}$, the whole expression simplifies to $\{ h (a)\ :\  a\in \bigcup\{A_{i} : i\in I\}\}$. In turn, 
\[ \epsilon_{\mathcal{B}}\circ \tilde{\mathbb{P}}\tilde{\mathbb{P}} h (\{A_{i}\}_{i\in I})=\epsilon_{\mathcal{B}}(\{\{ h (a)\ :\  a\in A_{i}\}\ :\  i\in I\}),\]
which is equal to
\[ \bigcup\{\{ h (a)\ :\  a\in A_{i}\} i\in I\}=\{ h (a)\ :\  a\in \bigcup\{A_{i} : i\in I\}\},\]
giving us the desired equality.
\end{proof}

\begin{theorem}
The triple of $\tilde{\mathbb{P}}$, $\eta$ and $\epsilon$ forms a monad.
\end{theorem}

\begin{proof}
Let $\mathcal{A}$ be a $\Sigma$-multialgebra. We first must prove $\epsilon\circ \tilde{\mathbb{P}}\epsilon=\epsilon\circ\epsilon \tilde{\mathbb{P}}$, what amounts to $\epsilon_{\mathcal{A}}\circ \tilde{\mathbb{P}}\epsilon_{\mathcal{A}}=\epsilon_{\mathcal{A}}\circ \epsilon_{\tilde{\mathbb{P}}\mathcal{A}}$, as homomorphisms from $\tilde{\mathbb{P}}^{3}\mathcal{A}$ to $\tilde{\mathbb{P}}\mathcal{A}$. So, let $\{\{A_{i}^{j}\}_{i\in I}\}_{j\in J}$ be an element of $\tilde{\mathbb{P}}^{3}\mathcal{A}$, where $I$ and $J$ are non-empty sets of indexes and all $A_{i}^{j}$ are non-empty subsets of $A$: 
\[ \epsilon_{\mathcal{A}}\circ \tilde{\mathbb{P}}\epsilon_{\mathcal{A}}(\{\{A_{i}^{j}\}_{i\in I}\}_{j\in J})=\epsilon_{\mathcal{A}}(\{\epsilon_{\mathcal{A}}(\{A_{i}^{j}\ :\  i\in I\})\ :\  j\in J\})=\epsilon_{\mathcal{A}}(\{\bigcup\{A_{i}^{j} : i\in I\}\ :\  j\in J\}),\]
what equals $\bigcup\{\bigcup\{A_{i}^{j} : i\in I\} : j\in J\}$, while 
\[\epsilon_{\mathcal{A}}\circ\epsilon_{\tilde{\mathbb{P}}\mathcal{A}}(\{\{A_{i}^{j}\}_{i\in I}\}_{j\in J})=\epsilon_{\mathcal{A}}(\bigcup\{\{A_{i}^{j} : j\in J\}\}_{i\in I})=\bigcup\{\bigcup\{A_{i}^{j} : j\in J\} : i\in I\},\]
and it is clear that both sets are the same.

It remains to be proven $\epsilon\circ \tilde{\mathbb{P}}\eta=\epsilon\circ\eta \tilde{\mathbb{P}}=1_{\tilde{\mathbb{P}}}$, meaning that $\epsilon_{\mathcal{A}}\circ \eta_{\tilde{\mathbb{P}}\mathcal{A}}=\epsilon_{\mathcal{A}}\circ \tilde{\mathbb{P}}\eta_{\mathcal{A}}$, as homomorphisms from $\tilde{\mathbb{P}}\mathcal{A}$ to $\tilde{\mathbb{P}}\mathcal{A}$, and this coincides with the identity homomorphism on this multialgebra as well. So, we take a non-empty subset $A'$ of $A$, and we have that $\epsilon_{\mathcal{A}}\circ\eta_{\tilde{\mathbb{P}}\mathcal{A}}(A')=\epsilon_{\mathcal{A}}(\{A'\})=A'$, while for the other expression one derives
\[ \epsilon_{\mathcal{A}}\circ P\eta_{\mathcal{A}}(A')=\epsilon_{\mathcal{A}}(\{\eta_{\mathcal{A}}(a)\ :\  a\in A'\})=\epsilon_{\mathcal{A}}(\{\{a\}\ :\  a\in A'\})=\bigcup\{\{a\} : a\in A'\}=A',\]
what finishes the proof.
\end{proof}

\subsection{Multialgebras of atoms}

Given a $(\Sigma, \leq)$-algebra $\mathcal{A}$, take the set $\mathbb{A}((A, \leq_{\mathcal{A}}))$ of atoms of $(A,\leq_{\mathcal{A}})$, that is, the set of minimal elements of this partially ordered set (equal to $A_{1}$ as well). For a $\sigma\in\Sigma_{n}$ and atoms $a_{1}, \ldots  , a_{n}\in \mathbb{A}((A, \leq_{\mathcal{A}}))$, we define 
\[ \sigma_{\mathbb{A}(\mathcal{A})}(a_{1}, \ldots  , a_{n})=\{a\in \mathbb{A}((A, \leq_{\mathcal{A}})):  a\leq_{\mathcal{A}} \sigma_{\mathcal{A}}(a_{1}, \ldots  , a_{n})\}=A_{\sigma_{\mathcal{A}}(a_{1}, \ldots  , a_{n})}.\]
This way, $(\mathbb{A}((A, \leq_{\mathcal{A}})), \{\sigma_{\mathbb{A}(\mathcal{A})}\}_{\sigma\in\Sigma})$ becomes a $\Sigma$-multialgebra, that we will denote by $\mathbb{A}(\mathcal{A})$ and call the multialgebra of atoms of $\mathcal{A}$. Given $(\Sigma, \leq)$-algebras $\mathcal{A}$ and $\mathcal{B}$ and a $(\Sigma, \leq)$-homomorphism $ h :\mathcal{A}\rightarrow \mathcal{B}$, we also define $\mathbb{A}( h ):\mathbb{A}((A, \leq_{\mathcal{A}}))\rightarrow\mathbb{A}((B, \leq_{\mathcal{B}}))$ as the restriction of $ h $ to $\mathbb{A}((A, \leq_{\mathcal{A}}))\subseteq A$. It is well-defined since every $(\Sigma, \leq)$-homomorphism preserves minimal elements, that is, atoms.

For $\sigma\in\Sigma_{n}$ and atoms $a_{1}, \ldots  , a_{n}\in \mathbb{A}((A, \leq_{\mathcal{A}}))$ we have that 
\[ \{\mathbb{A}( h )(a):  a\in \sigma_{\mathbb{A}(\mathcal{A})}(a_{1}, \ldots  , a_{n})\}=\{ h (a):  a\in\sigma_{\mathbb{A}(\mathcal{A})}(a_{1}, \ldots  , a_{n})\}=\]
\[\{ h (a)\in \mathbb{A}((B, \leq_{\mathcal{B}})):  a\leq_{\mathcal{A}} \sigma_{\mathcal{A}}(a_{1}, \ldots  , a_{n})\}\]
and, since $a\leq_{\mathcal{A}} \sigma_{\mathcal{A}}(a_{1}, \ldots  , a_{n})$ implies $ h (a)\leq_{\mathcal{B}} h (\sigma_{\mathcal{A}}(a_{1}, \ldots  , a_{n}))$ given $ h $ is order preserving, which in turn implies $ h (a)\leq_{\mathcal{B}}\sigma_{\mathcal{B}}( h (a_{1}), \ldots  ,  h (a_{n}))$ since $ h $ is an ``almost homomorphism'', we get that
\[ \{ h (a)\in \mathbb{A}((B, \leq_{\mathcal{B}})):  a\leq_{\mathcal{A}} \sigma_{\mathcal{A}}(a_{1}, \ldots  , a_{n})\}\subseteq\{b\in \mathbb{A}((B, \leq_{\mathcal{B}})):  b\leq_{\mathcal{B}}\sigma_{\mathcal{B}}( h (a_{1}), \ldots  ,  h (a_{n})\}=\]
\[ \sigma_{\mathbb{A}(\mathcal{B})}( h (a_{1}), \ldots  ,  h (a_{n}))=\sigma_{\mathbb{A}(\mathcal{B})}(\mathbb{A}( h )(a_{1}), \ldots  , \mathbb{A}( h )(a_{n})),\]
what proves $\mathbb{A}( h )$ is a homomorphism between $\Sigma$-multialgebras, and we may write $\mathbb{A}( h ):\mathbb{A}(\mathcal{A})\rightarrow\mathbb{A}(\mathcal{B})$.

The natural question is if $\mathbb{A}:\textbf{Alg}_{\mathsf{B}}(\Sigma)\rightarrow\textbf{MAlg}(\Sigma)$ is a functor, to which the answer is yes: it is easy to see that it distributes over the composition of morphisms and preserves the identical ones.

\section{$\textbf{Alg}_{\mathsf{B}}(\Sigma)$ and $\textbf{MAlg}(\Sigma)$ are equivalent} \label{sect-equiv}

Now, we aim to prove that $\textbf{Alg}_{\mathsf{B}}(\Sigma)$ and $\textbf{MAlg}(\Sigma)$ are actually equivalent categories, the equivalence being given by the functors $\mathbb{P}$ and $\mathbb{A}$. In order to prove that $\mathbb{P}$ and $\mathbb{A}$ form an equivalence of categories it is enough proving that both are full and faithful and $\mathbb{A}$ is a right adjoint of $\mathbb{P}$.

\subsection{$\mathbb{P}$ and $\mathbb{A}$ are full and faithful}

It is easy to see $\mathbb{P}$ is faithful: given $\Sigma$-multialgebras $\mathcal{A}$ and $\mathcal{B}$, and homomorphisms $ h ,  h' :\mathcal{A}\rightarrow\mathcal{B}$, if $\mathbb{P}( h )=\mathbb{P}( h' )$, we have that, for every $a\in A$,
\[ \{ h (a)\}=\mathbb{P}( h )(\{a\})=\mathbb{P}( h' )(\{a\})=\{ h' (a)\},\]
and therefore $ h = h' $.

\begin{proposition}
$\mathbb{A}$ is faithful.
\end{proposition}

\begin{proof}
Given $(\Sigma, \leq)$-algebras $\mathcal{A}$ and $\mathcal{B}$, and $(\Sigma, \leq)$-homomorphisms $ h ,  h' :\mathcal{A}\rightarrow\mathcal{B}$, suppose that $\mathbb{A}( h )=\mathbb{A}( h' )$. Then, for every $a\in A$, we can write $a=\sup A_{a}$, since $(A, \leq_{\mathcal{A}})$ is atomic.

Since $ h $ and $ h' $ are continuous, $ h (a)=\sup\{ h (a'):  a'\in A_{a}\}$ and $ h' (a)=\sup\{ h' (a'):  a'\in A_{a}\}$. But, since $\mathbb{A}( h )=\mathbb{A}( h' )$, $ h $ and $ h' $ are the same when restricted to atoms, and therefore $\{ h (a'):  a'\in A_{a}\}=\{ h' (a'):  a'\in A_{a}\}$. This means that $ h (a)= h' (a)$ and, since $a$ is arbitrary, $ h = h' $.
\end{proof}

Now, given $\Sigma$-multialgebras $\mathcal{A}$ and $\mathcal{B}$, and a $(\Sigma, \leq)$-homomorphism $ h :\mathbb{P}(\mathcal{A})\rightarrow\mathbb{P}(\mathcal{B})$, to prove that $\mathbb{P}$ is also full we must find a homomorphism $ h' :\mathcal{A}\rightarrow\mathcal{B}$ such that $\mathbb{P}( h' )= h $.

For every $a\in A$, $\{a\}$ is an atom and, since $ h $ preserves atoms, $ h (\{a\})$ is an atom of $\mathbb{P}(B)$, and therefore of the form $\{b_{a}\}$ for some $b_{a}\in B$. We define $ h' :\mathcal{A}\rightarrow\mathcal{B}$ by $ h' (a)=b_{a}$. First of all, we must show $ h' $ is in fact a homomorphism, which is quite analogous to the proof of the same fact for $\mathbb{A}( h )$. Given $\sigma\in\Sigma_{n}$ and $a_{1}, \ldots  , a_{n}\in A$,
\[ \{ h' (a):  a\in \sigma_{\mathcal{A}}(a_{1}, \ldots  , a_{n})\}=\{b_{a}:  a\in\sigma_{\mathcal{A}}(a_{1}, \ldots  , a_{n})\}=\sup\{\{b_{a}\}:  a\in\sigma_{\mathcal{A}}(a_{1}, \ldots  , a_{n})\}=\]
\[\sup\{ h (\{a\}):  a\in\sigma_{\mathcal{A}}(a_{1}, \ldots  , a_{n})\}=h (\sup \{\{a\}:  a\in\sigma_{\mathcal{A}}(a_{1}, \ldots  , a_{n})\}),\]
given that  $ h $ is continuous, and since it is an ``almost homomorphism'' this equals
\[  h (\sigma_{\mathcal{A}}(a_{1}, \ldots  , a_{n}))= h (\sigma_{\mathbb{P}(\mathcal{A})}(\{a_{1}\}, \ldots  , \{a_{n}\}))\subseteq \sigma_{\mathbb{P}(\mathcal{B})}( h (\{a_{1}\}), \ldots  ,  h (\{a_{n}\}))=\]
\[ \sigma_{\mathbb{P}(\mathcal{B})}(\{b_{a_{1}}\}, \ldots  , \{b_{a_{n}}\})=\sigma_{\mathcal{B}}(b_{a_{1}}, \ldots  , b_{a_{n}})=\sigma_{\mathcal{B}}( h' (a_{1}), \ldots  ,  h' (a_{n})).\]
Now, when we consider $\mathbb{P}( h' )$, we see that, for every atom $\{a\}$ of $\mathbb{P}(\mathcal{A})$, $\mathbb{P}( h' )(\{a\})=\{b_{a}\}= h (\{a\})$, and so the restrictions of $ h $ and $\mathbb{P}( h' )$ to atoms are the same, and therefore $\mathbb{A}( h )=\mathbb{A}(\mathbb{P}( h' ))$. Since $\mathbb{A}$ is faithful, we discover that $ h =\mathbb{P}( h' )$ and, as we stated before, $\mathbb{P}$ is full.

Now it remains to be shown that $\mathbb{A}$ is also full. Given $(\Sigma, \leq)$-algebras $\mathcal{A}$ and $\mathcal{B}$, and a homomorphism $ h :\mathbb{A}(\mathcal{A})\rightarrow\mathbb{A}(\mathcal{B})$, we then define $ h' :\mathcal{A}\rightarrow\mathcal{B}$ by
\[  h' (a)=\sup\{ h (c):  c\in A_{a}\}.\]
First of all, we must prove that $ h' $ is a $(\Sigma, \leq)$-homomorphism, for which we shall need a few lemmas.

\begin{lemma}\label{supsup equals supunion}
In a complete, atomic and bottomless Boolean algebra $\mathcal{A}$, take a non-empty family of indexes $I$ and, for every $i\in I$, $X_{i}\subseteq A$. Suppose we have $x_{i}=\sup X_{i}$, for $i\in I$, and $X=\bigcup\{X_{i} : i\in I\}$. Then, $\sup\{x_{i}:  i\in I\}=\sup X$.
\end{lemma}

\begin{proof}
We define $a=\sup\{x_{i}:  i\in I\}$ and $b=\sup X$: first, we show that $a$ is an upper bound for $X$, so that $a\geq_{\mathcal{A}} b$. For every $x\in X$, we have that, since $X=\bigcup\{X_{i} : i\in I\}$, there exists $j\in I$ such that $x\in X_{j}$, and therefore $x_{j}\geq_{\mathcal{A}} x$. Since $a=\sup\{x_{i}:  i\in I\}$, we have that $a\geq_{\mathcal{A}} x_{j}$, and by transitivity $a\geq_{\mathcal{A}} x$, and therefore $a$ is indeed an upper bound for $X$.

Now we show that $b$ is an upper bound for $\{x_{i}:  i\in I\}$, and so $b\geq_{\mathcal{A}} a$ (and $a=b$). For every $i\in I$, we have that $b$ is an upper bound for $X_{i}$, since $X_{i}\subseteq X$ and $b$ is an upper bound for $X$, and therefore $b\geq_{\mathcal{A}} x_{i}$, since $x_{i}$ is the smallest upper bound for $X_{i}$. It follows that $b$ is indeed an upper bound for $\{x_{i}:  i\in I\}$, what finishes the proof.
\end{proof}

\begin{lemma}\label{union atoms is atoms sup}
In a complete, atomic and bottomless Boolean algebra $\mathcal{A}$, for a non-empty $C\subseteq A$ one has that $\bigcup\{A_{c} : c\in C\}=A_{\sup C}$.
\end{lemma}

\begin{proof}
If $d\in A_{c}$ for a $c\in C$, $d$ is an atom such that $d \leq_{\mathcal{A}} c$. Since $c\leq_{\mathcal{A}}\sup C$, $d\leq_{\mathcal{A}}\sup C$, and therefore $d$ belongs to $A_{\sup C}$. Thus $\bigcup\{A_{c} : c\in C\}\subseteq A_{\sup C}$.

Reciprocally, suppose that $d\in A_{\sup C}$. Then, $d$ is an atom such that $d\leq_{\mathcal{A}}\sup C$, and therefore $\inf\{d, \sup C\}=d$. It follows that the subset $C^{d}\subseteq C$, of $c\in C$ such that $\inf\{d,c\}$ exists, is not empty, by Proposition \ref{sem-inf-dist}. But if $c\in C^{d}$, $\inf\{d,c\}$ exists, and since $d$ is an atom, we have that $d\in A_{c}\subseteq \bigcup\{A_{c} : c\in C\}$, and from that $\bigcup\{A_{c} : c\in C\}=A_{\sup C}$.
\end{proof}

Since $\sigma_{\mathcal{A}}(a_{1}, \ldots  , a_{n})$ is equal to the supremum of $\{\sigma_{\mathcal{A}}(c_{1}, \ldots  , c_{n}):  (c_{1}, \ldots  , c_{n})\in A_{a_{1}}\times\cdots\times A_{a_{n}}\}$, from Lemma \ref{union atoms is atoms sup} we have that $A_{\sigma_{\mathcal{A}}(a_{1}, \ldots  , a_{n})}$ is equal to 
\[ \bigcup\{ A_{\sigma_{\mathcal{A}}(c_{1}, \ldots  , c_{n})} : c_{1}\in A_{a_{1}}, \ldots  , c_{n}\in A_{a_{n}}\},\]
that is, we have the following lemma.

\begin{lemma}\label{atoms operations equals union atoms}
For a $\sigma\in\Sigma_{n}$, and $a_{1}, \ldots  , a_{n}\in A$,
\[ A_{\sigma_{\mathcal{A}}(a_{1}, \ldots  , a_{n})}=\bigcup\{A_{\sigma_{\mathcal{A}}(c_{1}, \ldots  , c_{n})} : c_{i}\in A_{a_{i}}\}.\]
\end{lemma}

\begin{theorem}
$\mathbb{A}$ is full.
\end{theorem}

\begin{proof}
So, first of all, we prove that $ h' $ is a $(\Sigma, \leq)$-homomorphism.

\begin{enumerate}
\item First, it is clear that $ h' $ maps atoms into atoms: if $a$ is an atom, $A_{a}=\{a\}$ and
\[  h' (a)=\sup\{ h (c):  c\in A_{a}\}=\sup\{ h (a)\}= h (a),\]
which is an atom since $ h $ is a map between $\mathbb{A}(\mathcal{A})$ and $\mathbb{A}(\mathcal{B})$.

\item $ h' $ is continuous: for any non-empty set $C\subseteq A$, we remember that $ h' (c)=\sup\{ h (d):  d\in A_{c}\}$, and from Lemmas \ref{supsup equals supunion} and \ref{union atoms is atoms sup} we get that
\[ \sup\{ h' (c):  c\in C\}=\sup\{\sup\{ h (d):  d\in A_{c}\}:  c\in C\}=\]
\[ \sup \bigcup\{\{ h (d):  d\in A_{c}\} : c\in C\}=\sup\{ h (d):  d\in A_{\sup C}\}= h' (\sup C).\]

\item Since $\{ h (a):  a\in\sigma_{\mathbb{A}(\mathcal{A})}(a_{1}, \ldots  , a_{n})\}\subseteq \sigma_{\mathbb{A}(\mathcal{B})}( h (a_{1}), \ldots  ,  h (a_{n}))$, given that $ h $ is a homomorphism of multialgebras, it follows from Lemma \ref{atoms operations equals union atoms} that
\[  h' (\sigma_{\mathcal{A}}(a_{1}, \ldots  , a_{n}))=\sup\{ h (c):  c\in \bigcup\{A_{\sigma_{\mathcal{A}}(c_{1}, \ldots  , c_{n})} : c_{i}\in A_{a_{i}}\}\}=\]
\[ \sup\bigcup\{\{ h (c):  c\in \sigma_{\mathbb{A}(\mathcal{A})}(c_{1}, \ldots  , c_{n})\} : c_{i}\in A_{a_{i}}\}\leq_{\mathcal{B}}\]
\[ \sup \bigcup\{\sigma_{\mathbb{A}(\mathcal{B})}( h (c_{1}), \ldots  ,  h (c_{n})) : c_{i}\in A_{a_{i}}\},\]
where we have used that, for atoms $c_{1}, \ldots  , c_{n}$ of $\mathcal{A}$, $\sigma_{\mathbb{A}(\mathcal{A})}(c_{1}, \ldots  , c_{n})=A_{\sigma_{\mathcal{A}}(c_{1}, \ldots  , c_{n})}$; since, for atoms $d_{1}, \ldots  , d_{n}$ of $\mathcal{B}$, we also have that $\sigma_{\mathbb{A}(\mathcal{B})}(d_{1}, \ldots  , d_{n})=A_{\sigma_{\mathcal{B}}(d_{1}, \ldots  , d_{n})}$, this is equal to
\[ \sup \bigcup\{A_{\sigma_{\mathcal{B}}( h (c_{1}), \ldots  ,  h (c_{n}))} : c_{i}\in A_{a_{i}}\}=\sup \bigcup\{A_{\sigma_{\mathcal{B}}( h' (c_{1}), \ldots  ,  h' (c_{n}))} : c_{i}\in A_{a_{i}}\}.\]
Since $ h' $ is continuous, $c_{i}\leq_{\mathcal{A}}a_{i}$, for every $i\in \{1, \ldots  , n\}$, implies $ h' (c_{i})\leq_{\mathcal{B}} h' (a_{i})$, and therefore\\ $\sigma_{\mathcal{B}}( h' (c_{1}), \ldots  ,  h' (c_{n}))\leq_{\mathcal{B}}\sigma_{\mathcal{B}}( h' (a_{1}), \ldots  ,  h' (a_{n}))$ for $(c_{1}, \ldots  , c_{n})\in A_{a_{1}}\times\cdots\times A_{a_{n}}$. It follows that the union, for $(c_{1}, \ldots  , c_{n})$ in $A_{a_{1}}\times\cdots\times A_{a_{n}}$, of $A_{\sigma_{\mathcal{B}}( h' (c_{1}), \ldots  ,  h' (c_{n}))}$, is contained on $A_{\sigma_{\mathcal{B}}( h' (a_{1}), \ldots  ,  h' (a_{n}))}$, and therefore 
\[ \sup \bigcup\{A_{\sigma_{\mathcal{B}}( h' (c_{1}), \ldots  ,  h' (c_{n}))} : c_{i}\in A_{a_{i}}\}\leq_{\mathcal{B}}\sup A_{\sigma_{\mathcal{B}}( h' (a_{1}), \ldots  ,  h' (a_{n}))}=\sigma_{\mathcal{B}}( h' (a_{1}), \ldots  ,  h' (a_{n})).\]
\end{enumerate}
Now, for every atom $a$ of $\mathcal{A}$, we have that $ h' (a)= h (a)$, and therefore the restriction of $h'$ to atoms coincides with $h$, that is, $\mathbb{A}( h' )= h $, and since $ h $ was taken arbitrarily, $\mathbb{A}$ is full.
\end{proof}

\subsection{$\mathbb{P}$ and $\mathbb{A}$ are adjoint}

It remains to be shown that $\mathbb{P}$ and $\mathbb{A}$ are adjoint. To this end, consider the bijections
\[ \Phi_{\mathcal{B}, \mathcal{A}}:Hom_{\textbf{MAlg}(\Sigma)}(\mathbb{A}(\mathcal{B}), \mathcal{A})\rightarrow Hom_{\textbf{Alg}_{\mathsf{B}}(\Sigma)}(\mathcal{B}, \mathbb{P}(\mathcal{A})),\]
for $\mathcal{A}$ a $\Sigma$-multialgebra and $\mathcal{B}$ a $(\Sigma, \leq)$-algebra, given by, for $ h :\mathbb{A}(\mathcal{B})\rightarrow\mathcal{A}$ a homomorphism and $b$ an element of $\mathcal{B}$,
\[ \Phi_{\mathcal{B},\mathcal{A}}( h )(b)=\{ h (c):  c\in A_{b}\}.\]

\begin{proposition}
$\Phi_{\mathcal{B}, \mathcal{A}}( h )$ is a $(\Sigma, \leq)$-homomorphism.
\end{proposition}

\begin{proof}
\begin{enumerate}
\item If $b$ is an atom, $A_{b}=\{b\}$, and therefore $\Phi_{\mathcal{B}, \mathcal{A}}( h )(b)=\{ h (c):  c\in A_{b}\}=\{ h (b)\}$, which is a singleton and therefore an atom of $\mathbb{P}(\mathcal{A})$.

\item Let $D$ be a non-empty subset of $\mathcal{B}$. We have that
\[ \Phi_{\mathcal{B}, \mathcal{A}}( h )(\sup D)=\{ h (c):  c\in A_{\sup D}\}=\{ h (c):  c\in \bigcup\{A_{d} : d\in D\}\}=\]
\[\bigcup\{\{ h (c):  c\in A_{d}\} : d\in D\}=\bigcup\{\Phi_{\mathcal{B}, \mathcal{A}}( h )(d) : d\in D\}=\sup\{\Phi_{\mathcal{B}, \mathcal{A}}( h )(d):  d\in D\},\]
since $A_{\sup D}=\bigcup_{d\in D}A_{d}$ and the supremum in $\mathbb{P}(\mathcal{A})$ is simply the union.

\item For $\sigma\in\Sigma_{n}$ and $b_{1}, \ldots  , b_{n}$ elements of $\mathcal{B}$, we have that
\[ \Phi_{\mathcal{B}, \mathcal{A}}( h )(\sigma_{\mathcal{B}}(b_{1}, \ldots  , b_{n}))=\{ h (c):  c\in A_{\sigma_{\mathcal{B}}(b_{1}, \ldots  , b_{n})}\}=\{ h (c):  c\in \bigcup\{ A_{\sigma_{\mathcal{B}}(c_{1}, \ldots  , c_{n})} : c_{i}\in A_{b_{i}}\}\}=\]
\[ \bigcup\{\{ h (c):  c\in A_{\sigma_{\mathcal{B}}(c_{1}, \ldots  , c_{n})}\} : c_{i}\in A_{b_{i}}\},\]
and, since $c_{1}, \ldots  , c_{n}$ are atoms, this is equal to
\[ \bigcup\{\{ h (c):  c\in \sigma_{\mathbb{A}(\mathcal{B})}(c_{1}, \ldots  , c_{n})\} : c_{i}\in A_{b_{i}}\}\subseteq \bigcup\{\sigma_{\mathcal{A}}( h (c_{1}), \ldots  ,  h (c_{n})) : c_{i}\in A_{b_{i}}\}=\]
\[ \bigcup\{\sigma_{\mathcal{A}}(a_{1}, \ldots  , a_{n}) : a_{i}\in \{ h  : c\in A_{b_{i}}\}\}=\bigcup\{\sigma_{\mathcal{A}}(a_{1}, \ldots  , a_{n}) : a_{i} \in \Phi_{\mathcal{B}, \mathcal{A}}( h )(b_{i})\}=\]
\[ \sigma_{\mathbb{P}(\mathcal{A})}(\Phi_{\mathcal{B}, \mathcal{A}}( h )(b_{1}), \ldots  , \Phi_{\mathcal{B}, \mathcal{A}}( h )(b_{n})).\]
\end{enumerate}
\end{proof}

Now, the $\Phi_{\mathcal{B}, \mathcal{A}}$ must be bijections between $Hom_{\textbf{MAlg}(\Sigma)}(\mathbb{A}(\mathcal{B}), \mathcal{A})$ and $Hom_{\textbf{Alg}_{\mathsf{B}}(\Sigma)}(\mathcal{B}, \mathbb{P}(\mathcal{A}))$. They are certainly injective: if $\Phi_{\mathcal{B}, \mathcal{A}}( h )=\Phi_{\mathcal{B}, \mathcal{A}}( h' )$, for every atom $b$ we have that
\[ \{ h (b)\}=\Phi_{\mathcal{B}, \mathcal{A}}( h )(b)=\Phi_{\mathcal{B}, \mathcal{A}}( h' )(b)=\{ h' (b)\},\]
and therefore $ h = h' $.

For the surjectivity, take a $(\Sigma, \leq)$-homomorphism $ h :\mathcal{B}\rightarrow \mathbb{P}(\mathcal{A})$. We then define $ h' :\mathbb{A}(\mathcal{B})\rightarrow \mathcal{A}$ by $ h' (b)=a$ for an atom $b$ in $\mathcal{B}$, where $ h (b)=\{a\}$. It is well-defined since a $(\Sigma, \leq)$-homomorphism takes atoms to atoms, and the atoms of $\mathbb{P}(\mathcal{A})$ are exactly the singletons.

We must show that $ h' $ is truly a homomorphism. For $\sigma\in\Sigma_{n}$ and atoms $b_{1}, \ldots  , b_{n}$ in $\mathbb{A}(\mathcal{B})$ such that $ h (b_{i})=\{a_{i}\}$ for every $i\in\{1, \ldots  , n\}$, we have that $ h (\sigma_{\mathcal{B}}(b_{1}, \ldots  , b_{n}))\subseteq\sigma_{\mathbb{P}(\mathcal{A})}( h (b_{1}), \ldots  ,  h (b_{n}))$, since $ h $ is a $(\Sigma, \leq)$-homomorphism, and therefore
\[ \{ h' (b):  b\in\sigma_{\mathbb{A}(\mathcal{B})}(b_{1}, \ldots  , b_{n})\}=\{ h' (b):  b\in A_{\sigma_{\mathcal{B}}(b_{1}, \ldots  , b_{n})}\}= \bigcup\{ h (b) : b\in A_{\sigma_{\mathcal{B}}(b_{1}, \ldots  , b_{n})}\}=\]
\[ h (\sup A_{\sigma_{\mathcal{B}}(b_{1}, \ldots  , b_{n})})= h (\sigma_{\mathcal{B}}(b_{1}, \ldots  , b_{n}))\subseteq \sigma_{\mathbb{P}(\mathcal{A})}( h (b_{1}), \ldots  ,  h (b_{n}))=\sigma_{\mathbb{P}(\mathcal{A})}(\{a_{1}\}, \ldots  , \{a_{n}\})=\]
\[ \sigma_{\mathcal{A}}(a_{1}, \ldots  , a_{n})=\sigma_{\mathcal{A}}( h' (b_{1}), \ldots  ,  h' (b_{n})).\]

Finally, we state that $\Phi_{\mathcal{B}, \mathcal{A}}( h' )= h $ since, for any element $b$ in $\mathcal{B}$, we have that
\[ \Phi_{\mathcal{B}, \mathcal{A}}( h' )(b)=\{ h' (c):  c\in A_{b}\}=\bigcup\{ h (c) : c\in A_{b}\}= h (\sup A_{b})= h (b),\]
and therefore the $\Phi_{\mathcal{B}, \mathcal{A}}$ are, indeed, bijective.

Given $\mathcal{A}$ and $\mathcal{C}$ two $\Sigma$-multialgebras, $\mathcal{B}$ and $\mathcal{D}$ two $(\Sigma, \leq)$-algebras, $ h :\mathcal{A}\rightarrow\mathcal{C}$ a homomorphism and $ h' :\mathcal{D}\rightarrow\mathcal{B}$ a $(\Sigma, \leq)$-homomorphism, we must now only prove that the following diagram commutes.

\[ \xymatrix{Hom_{\textbf{MAlg}(\Sigma)}(\mathbb{A}(\mathcal{B}), \mathcal{A}) \ar[rr]^{\Phi_{\mathcal{B}, \mathcal{A}}} \ar[dd]^{Hom(\mathbb{A}( h' ),  h )} & & Hom_{\textbf{Alg}_{\mathsf{B}}(\Sigma)}(\mathcal{B}, \mathbb{P}(\mathcal{A})) \ar[dd]^{Hom( h' , \mathbb{P}( h ))} \\%
& & \\
Hom_{\textbf{MAlg}(\Sigma)}(\mathbb{A}(\mathcal{D}), \mathcal{C}) \ar[rr]^{\Phi_{\mathcal{D}, \mathcal{C}}} & & Hom_{\textbf{Alg}_{\mathsf{B}}(\Sigma)}(\mathcal{D}, \mathbb{P}(\mathcal{C}))
}\]

So, we take a homomorphism $g:\mathbb{A}(\mathcal{B})\rightarrow\mathcal{A}$ and an element $d$ of $\mathcal{D}$. We have that
\[ Hom( h' , \mathbb{P}( h ))(\Phi_{\mathcal{B}, \mathcal{A}}(g))=\mathbb{P}( h )\circ\Phi_{\mathcal{B}, \mathcal{A}}(g)\circ h' ,\]
and therefore the right side of the diagram gives us 
\[ \mathbb{P}( h )\circ\Phi_{\mathcal{B}, \mathcal{A}}(g)\circ h' (d)=\mathbb{P}( h )(\{g(b):  b\in A_{ h' (d)}\})=\{ h \circ g(b):  b\in A_{ h' (d)}\}.\]

The left side gives us 
\[ \Phi_{\mathcal{D}, \mathcal{C}}( h \circ g\circ \mathbb{A}( h' ))(d)=\{ h \circ g\circ \mathbb{A}( h' )(e):  e\in A_{d}\}=\{ h \circ g\circ h' (e):  e\in A_{d}\}.\]

If $d$ is an atom, the right side becomes the singleton containing only $ h \circ g\circ h' (d)$, since in this case $A_{d}=\{d\}$ and, given that $ h' $ preserves atoms, $A_{ h' (d)}=\{ h' (d)\}$. The left side becomes also the singleton formed by $ h \circ g\circ h' (d)$, because again $A_{d}=\{d\}$. As a $(\Sigma, \leq)$-homomorphism is determined by its action on atoms, we find that the left and right sides of the diagram are equal, and therefore the diagram commutes.

As observed before, this proves $\mathbb{A}$ and $\mathbb{P}$ are adjoint and, therefore, that $\textbf{MAlg}(\Sigma)$ and $\textbf{Alg}_{\mathsf{B}}(\Sigma)$ are equivalent.

\section{Some consequences and related results}

The result that $\textbf{MAlg}(\Sigma)$ and $\textbf{Alg}_{\mathsf{B}}(\Sigma)$ are equivalent has a few consequences, and related results, we would like to stress. First of all, we start by taking the empty signature: in that case, given all multialgebras are non-empty, $\textbf{MAlg}(\Sigma)$ becomes the category of non-empty sets $\textbf{Set}^{*}$, with functions between them as morphisms. 

Meanwhile, $\textbf{Alg}_{\mathsf{B}}(\Sigma)$ becomes the category with complete, atomic and bottomless Boolean algebras as objects (given we simply drop the operations from a $(\Sigma, \leq)$-algebra), with continuous, atoms-preserving functions between them as morphisms. Notice this is very closely related to the equivalence between $\textbf{CABA}$ and $\textbf{Set}^{op}$: the morphisms on the former are merely continuous functions, so the only extra requirement to the morphisms we are making is that they should preserve atoms. This, of course, allows one to exchange the opposite category of $\textbf{Set}$ by $\textbf{Set}$ itself (or rather $\textbf{Set}^{*}$).

A generalization of our result is to partial multialgebras. That is, pairs $\mathcal{A}=(A, \{\sigma_{\mathcal{A}}\}_{\sigma\in\Sigma})$ such that, if $\sigma\in\Sigma_{n}$, $\sigma_{\mathcal{A}}$ is a function from $A^{n}$ to $\mathcal{P}(A)$ (no longer $\mathcal{P}(A)\setminus\{\emptyset\}$). In other words, a partial multialgebra is a multialgebra where operations may return the empty-set. Given partial $\Sigma$-multialgebras $\mathcal{A}$ and $\mathcal{B}$, a homomorphism between them is a function $ h :A\rightarrow B$ such that, as is the case for homomorphisms for multialgebras,
\[ \{ h (a): a\in\sigma_{\mathcal{A}}(a_{1}, \ldots  , a_{n})\}\subseteq \sigma_{\mathcal{B}}( h (a_{1}), \ldots  ,  h (a_{n})),\]
for $\sigma\in\Sigma_{n}$ and $a_{1}, \ldots  , a_{n}\in A$. The class of all partial $\Sigma$-multialgebras, with these homomorphisms between them as morphisms, becomes a category, which we shall denote by $\textbf{PMAlg}(\Sigma)$.

It is easy to find an equivalence, much alike the one between $\textbf{MAlg}(\Sigma)$ and $\textbf{Alg}_{\mathsf{B}}(\Sigma)$, between $\textbf{PMAlg}(\Sigma)$ and a category related to $\textbf{Alg}_{\mathsf{B}}(\Sigma)$: it is sufficient to replace the requirement that, in a $(\Sigma, \leq)$-algebra, $(A, \leq_{\mathcal{A}})$ is a complete, atomic and bottomless Boolean algebra, by the requisite that it is actually a complete, atomic Boolean algebra, and accordingly, change the morphisms in the correspondent category by requiring they preserve the supremum of any sets, not necessarily non-empty.

Finally, consider a modified notion of homomorphism between $\Sigma$-multialgebras $\mathcal{A}$ and $\mathcal{B}$, that of a function $ h :A\rightarrow\mathcal{P}(B)\setminus\{\emptyset\}$ such that
\[ \bigcup_{a\in \sigma_{\mathcal{A}}(a_{1}, \ldots  , a_{n})} h (a)\subseteq \bigcup_{(b_{1}, \ldots  , b_{n})\in  h (a_{1})\times\cdots\times h (a_{n})}\sigma_{\mathcal{B}}(b_{1}, \ldots  , b_{n}).\]
The category with $\Sigma$-multialgebras as objects and these homomorphisms as morphisms will be denoted by\\ $\textbf{MMAlg}(\Sigma)$. If, in the category $\textbf{Alg}_{\mathsf{B}}(\Sigma)$, we change morphisms by not longer demanding that they map atoms into atoms, it is easy to adapt the proof given in Section~\ref{sect-equiv} to show that the resulting category is equivalent to $\textbf{MMAlg}(\Sigma)$.

\section*{Conclusion and Future Work}

As we explained before, the main results here presented are a generalization of the equivalence between the categories of complete, atomic Boolean algebras and $\textbf{Set}^{op}$. On the one hand, we add operations to Boolean algebras that are compatible with its order, while on the other we allow for non-deterministic operations taking us to a category of multialgebras.

Although not specially complicated, this result is useful as it allows to treat non-deterministic matrices (Nmatrices) as, not precisely algebraic semantics, but mixed methods that combine both an algebraic component and one relative to its order. This may seem to increase the complexity of decision methods, but this sacrifice is made precisely to avoid non-determinism and use merely classical concepts. This is made, not because we distrust the use of multialgebras as semantics for non-classical logics, but as an alternative to those logicians that have philosophical objections against that very use.

More pragmatically, we are encouraged to further study the categories of multialgebras, now from the viewpoint of categories of partially ordered sets, far better understood than the former ones; moreover, we can now recast several non-deterministic characterizations of logics found in the literature in the terms here presented. Specifically, there are several paraconsistent logics uncharacterizable by finite matrices, but characterized by finite Nmatrices, which can now have semantics presented only in classical terms of algebras and orders.

\paragraph{Acknowledgements}
The first author acknowledges financial support from the National Council for Scientific and Technological Development (CNPq), Brazil, under research grant 306530/2019-8. The second author was initially supported by a doctoral scholarship from the {\em Coordenação de Aperfei\c{c}oamento de Pessoal de N\'ivel Superior -- Brasil (CAPES)} -- Finance Code 001 (Brazil), and later on by a postdoctoral fellowship by NSF-BSF, grant number 2020704.

\end{document}